\documentclass[11pt,reqno]{amsart}

\usepackage{amsthm,amsmath,amssymb}
\usepackage[margin=1in]{geometry}
\usepackage[
  colorlinks=true,
  citecolor=green,
  linkcolor=red,
  urlcolor=black
]{hyperref}
\usepackage{xcolor}
\usepackage{microtype}
\hypersetup{
  pdftitle={Simultaneous popular polynomial differences in higher dimensions.},
  pdfsubject={Polynomial recurrence and directional arithmetic regularity}
}

\numberwithin{equation}{section}

\newtheorem{theorem}{Theorem}[section]
\newtheorem{proposition}[theorem]{Proposition}
\newtheorem{lemma}[theorem]{Lemma}

\theoremstyle{remark}
\newtheorem{remark}[theorem]{Remark}

\usepackage[dvipsnames]{xcolor}

\usepackage[color=Orange!50!white, textsize=scriptsize]{todonotes}

\newcommand{\F}{\mathbb F}
\newcommand{\E}{\mathbb E}
\newcommand{\T}{\mathbb T}
\newcommand{\N}{\mathbb N}
\newcommand{\Z}{\mathbb Z}
\newcommand{\Q}{\mathbb Q}
\newcommand{\cP}{\mathcal P}
\newcommand{\cV}{\mathcal V}
\newcommand{\str}{\mathrm{str}}
\newcommand{\psr}{\mathrm{psr}}
\newcommand{\sml}{\mathrm{sml}}
\newcommand{\bx}{\underline{x}}
\newcommand{\bh}{\underline{h}}
\newcommand{\1}{\mathbf 1}

\title{Simultaneous popular polynomial differences in higher dimensions}
\date{}

\author{Guo-Dong Hong}
\address{Department of Mathematics, California Institute of Technology, Pasadena, CA 91125, USA}
\email{ghong@caltech.edu}

\begin{document}

\begin{abstract}
We extend the simultaneous popular polynomial difference theorem
of Conlon, Dong, and the author to polynomial shifts along
prescribed directions in $\mathbb{F}_p^D$.
Let $P_1,\ldots,P_k\in\mathbb{Z}[t]$ be linearly independent
polynomials with zero constant terms, and let
$v_1,\ldots,v_k\in\mathbb{Z}^D\setminus\{0\}$. We show that,
for every $\varepsilon>0$, every sufficiently large prime $p$,
and every $A\subseteq\mathbb{F}_p^D$ of density $\alpha$, there
exists a nonzero $d\in\mathbb{F}_p$ such that
\[
  \mathbb{E}_{\underline{x}\in\mathbb{F}_p^D}
  \mathbf{1}_A(\underline{x})
  \prod_{\substack{1\leq i\leq k\\ \omega_i=1}}
  \mathbf{1}_A\bigl(\underline{x}+P_i(d)v_i\bigr)
  \geq \alpha^{\,1+\sum_{i=1}^k\omega_i}-\varepsilon
\]
simultaneously for every $\omega\in\{0,1\}^k$.
The main new ingredient is a directional arithmetic regularity
decomposition adapted to mixed directional box norms.
\end{abstract}

\maketitle

\section{Introduction}
\label{sec:introduction}

Popular polynomial difference theorems seek a nonzero parameter
for which the normalized count of a prescribed polynomial
configuration is at least its expected value in a random set
of the same density, up to a small error.
A simultaneous version asks for a single parameter that gives
these lower bounds for several polynomial configurations at once.
Conlon, Dong, and the author~\cite{CDH26} established such a result
over finite fields. In this paper, we extend their theorem to
configurations in $\F_p^D$, with polynomial shifts along prescribed
directions.

To state our result, fix $D,k\in\N$, and set $G:=\F_p^D$,
where $p$ is an odd prime.
Let $\mathcal V=\{v_1,\ldots,v_k\}$ be a collection of nonzero
vectors in $\Z^D$, and let $\mathcal P=\{P_1,\ldots,P_k\}$
be a collection of linearly independent polynomials in $\Z[t]$
with zero constant terms. For
$d\in\F_p$ and functions $f_0,\ldots,f_k\colon G\to\mathbb C$, define
\[
  I_{\cP,\cV}(f_0,\ldots,f_k)(d)
  :=
  \E_{\bx\in G}f_0(\bx)
  \prod_{i=1}^k f_i\bigl(\bx+P_i(d)v_i\bigr).
\]
When all the functions are $\1_A$, this average gives the normalized
count of polynomial configurations
$\bx,\bx+P_1(d)v_1,\ldots,\bx+P_k(d)v_k$
contained in $A$, with the shifts determined by the common parameter $d$.
We consider subconfigurations obtained by keeping $\bx$ and
any subset of the shifted points.
Given $\omega=(\omega_1,\ldots,\omega_k)\in\{0,1\}^k$, let
$(\cP,\cV)_\omega$ denote the subcollection of pairs $(P_i,v_i)$ for which
$\omega_i=1$, and set
\[
  I_{(\cP,\cV)_\omega}(f_0,\ldots,f_k)(d)
  :=
  \E_{\bx\in G}f_0(\bx)
  \prod_{\substack{1\leq i\leq k\\ \omega_i=1}}
  f_i\bigl(\bx+P_i(d)v_i\bigr).
\]
Here and throughout, an empty product is understood to be $1$.

With this notation, our main result is as follows.
\begin{theorem}
\label{thm:main}
Let $D,k\in\N$. Let $\mathcal V=\{v_1,\ldots,v_k\}$ be a
collection of nonzero vectors in $\Z^D$, and let
$\mathcal P=\{P_1,\ldots,P_k\}\subset\Z[t]$ be a collection
of linearly independent polynomials with zero constant terms.
For every $\varepsilon>0$, there exists
\(
  p_0=p_0(\varepsilon,\cP,\mathcal{V})
\)
such that the following holds. For every prime $p\geq p_0$ and every
$A\subseteq\F_p^D$ of density $\alpha$, there exists
$d\in\F_p\setminus\{0\}$ such that
\[
  I_{(\cP,\mathcal{V})_\omega}(\1_A,\ldots,\1_A)(d)
  \geq
  \alpha^{\,1+\sum_{i=1}^k\omega_i}-\varepsilon
\]
holds simultaneously for every $\omega=(\omega_1,\dots,\omega_k)\in\{0,1\}^k$.
\end{theorem}

The same nonzero parameter $d$ therefore works for every
choice of $\omega$. In particular, it gives the lower
bound $\alpha^{k+1}-\varepsilon$ for the full configuration and the
lower bound $\alpha^2-\varepsilon$ for each of the $k$ two-term
configurations. No linear-independence assumption is imposed on the
vectors $v_1,\ldots,v_k$.

The study of popular polynomial differences has roots in ergodic
theory. Frantzikinakis and Kra~\cite{FK06} proved a Khintchine-type
recurrence theorem for rationally independent polynomials, giving a
lower bound for recurrence along the full polynomial configuration.
In a related direction, Lyall and Magyar~\cite{LM11} investigated
a different strengthening of Khintchine's recurrence theorem,
which they called simultaneous polynomial recurrence. Their result
shows that a single parameter can satisfy the Khintchine-type lower
bound for each of several two-term polynomial configurations.
These results suggest asking whether a single parameter can give
the corresponding lower bounds for the full configuration and all
its subconfigurations simultaneously. Conlon, Dong, and the
author~\cite{CDH26} established this simultaneous conclusion over
$\F_p$. More recently, Shao and Zhang~\cite{SZ26} obtained simultaneous
popular difference results over $\F_p$ for certain polynomial
configurations with linearly dependent shift polynomials.
In a complementary direction, Theorem~\ref{thm:main} extends
the result of~\cite{CDH26} to polynomial shifts along prescribed
directions in $\F_p^D$.

The proof of Theorem~\ref{thm:main} follows the regularity approach
developed in~\cite{CDH26}. In the one-dimensional setting, a generalized von Neumann
estimate derived from Peluse's polynomial Szemer\'edi
theorem~\cite{P19} controls contributions involving a component
that is small in the $U^2$-norm. Tao's arithmetic regularity
decomposition~\cite{Tao14} then reduces the argument for an arbitrary
set to estimating a structured contribution and a small error.

In the multidimensional setting, Kuca's polynomial Szemer\'edi
theorem~\cite{Kuca24} yields a generalized von Neumann estimate in
terms of the mixed directional box norms
$\|\cdot\|_{G,\langle v_i\rangle}$; see
Lemma~\ref{lem:directional-vn}. To use this estimate, we need a
regularity decomposition adapted to each prescribed direction.
Proposition~\ref{prop:directional-regularity} provides these
decompositions and is the main new ingredient in the proof.
The resulting structured components need not coincide.
We construct them using symmetrized Bohr kernels so that the
convolution inequality in Lemma~\ref{lem:convolution-inequality}
gives the required lower bound for their joint contribution.

The remainder of the paper is organized as follows.
Section~\ref{sec:preliminaries} introduces the notation and records
the preliminary estimates. In Section~\ref{sec:simultaneous-proof},
we prove Theorem~\ref{thm:main} using the directional arithmetic
regularity decomposition.

\section{Notation and preliminaries}
\label{sec:preliminaries}

We collect the notation and the directional Fourier and Bohr-set
estimates used in the proof.

\subsection{Conventions}

All averages are normalized.  Thus, for a finite nonempty set $X$,
\[
  \E_{x\in X}f(x):=\frac1{|X|}\sum_{x\in X}f(x),
  \qquad
  \|f\|_{\mathcal L^q(X)}
  :=\left(\E_{x\in X}|f(x)|^q\right)^{1/q}
  \quad(1\leq q<\infty).
\]
We omit an averaging domain when it is evident from the variables.  For
example, $\E f$ denotes the normalized mean of $f$ on its domain.  For a
set $E$, the function $\1_E$ denotes its indicator.
The density of a subset $A\subseteq X$ is
$\E_{x\in X}\1_A(x)=|A|/|X|$. We also put
$\|f\|_{\mathcal L^\infty(X)}:=\max_{x\in X}|f(x)|$.
We say that $f$ is \emph{$1$-bounded} if $|f(x)|\leq1$ for every $x$.

Throughout the paper, $p$ denotes an odd prime. Let $D\in\N$,
and denote $G=\F_p^D$. Linear independence of polynomials in $\Z[t]$
is understood over $\Q$. Integer vectors and polynomials are reduced
modulo $p$ whenever they are used over $\F_p$.
For $v\in G$, write $\langle v\rangle:=\{tv\colon t\in\F_p\}$.
For $a\in\F_p$, write $e_p(a):=\exp(2\pi i a/p)$.
For $f\colon\F_p\to\mathbb C$, our normalized Fourier transform is
\[
  \widehat f(\xi):=\E_{x\in\F_p}f(x)e_p(-\xi x).
\]
For $\bx,\bh\in G$, define the multiplicative derivative
\[
  \triangle_{\bh}f(\bx)
  :=f(\bx)\overline{f(\bx+\bh)}.
\]
We write
$\triangle_{\bh_1,\ldots,\bh_s}f
:=\triangle_{\bh_1}\cdots\triangle_{\bh_s}f$
for the iterated multiplicative derivative, with the operators
applied by composition.

We use Landau's notation $O,o$, and Vinogradov's notation $\ll,\gg$.
A subscript records the data on which the
implicit constant may depend.

\subsection{Directional Fourier coefficients and box norms}

We introduce the Fourier coefficients and box norms used to
control pseudorandomness in each prescribed direction.

For a nonzero $v\in G$, $\xi\in\F_p$, and
$f\colon G\to\mathbb C$, define the Fourier coefficient of $f$ along $v$ by
\begin{equation}
  \widehat f(\bx;v;\xi)
  :=
  \E_{n\in\F_p}f(\bx+nv)e_p(-n\xi).
  \label{eq:directional-fourier}
\end{equation}
If $\bx'=\bx+sv$, then
\[
  \widehat f(\bx';v;\xi)=e_p(s\xi)\widehat f(\bx;v;\xi),
\]
so the modulus of \eqref{eq:directional-fourier} is constant on cosets of
$\langle v\rangle$.

For subspaces $H_1,\ldots,H_s\leq G$, define the Gowers box norm along
$H_1,\ldots,H_s$ by
\[
  \|f\|_{H_1,\ldots,H_s}
  :=
  \left(
    \E_{\bx\in G}
    \E_{\bh_1\in H_1}\cdots\E_{\bh_s\in H_s}
    \triangle_{\bh_1,\ldots,\bh_s}f(\bx)
  \right)^{1/2^s}.
\]
The usual Gowers $U^2$-norm is the special case
$\|f\|_{U^2(G)}:=\|f\|_{G,G}$.
The notation $\|f\|_{G,\langle v\rangle}$ denotes the box norm with
$H_1=G$ and $H_2=\langle v\rangle$.  This is the mixed directional box
norm used in \cite{Hong25}.

The following estimate relates the directional Fourier coefficients
to the box norm appearing in the counting estimate; see
\cite[Lemma~2.2]{Hong25}.

\begin{lemma}
\label{lem:directional-fourier-control}
Let $v\in G\setminus\{0\}$ and $f\colon G\to\mathbb C$.  If $f$ is $1$-bounded, then
\begin{equation}
  \|f\|_{G,\langle v\rangle}^{4}
  \leq
  \sup_{\xi\in\F_p}
  \E_{\bx\in G}|\widehat f(\bx;v;\xi)|^2.
\end{equation}
\end{lemma}

\subsection{Bohr sets}

Bohr sets describe the shifts under which the structured components
are almost invariant. We record their definition and a standard lower
bound for their size.

For $t\in\T:=\mathbb R/\mathbb Z$, write $\|t\|_{\T}$ for the distance to
the nearest integer.  If $\Gamma\subseteq\F_p\setminus\{0\}$ and
$0<\rho\leq1/2$, define
\[
  B(\Gamma,\rho)
  :=
  \left\{
    b\in\F_p\colon
    \left\|\frac{\xi b}{p}\right\|_{\T}\leq\rho
    \text{ for every }\xi\in\Gamma
  \right\}.
\]
Its rank is $|\Gamma|$, and we adopt the convention
$B(\varnothing,\rho)=\F_p$.

We use the following standard estimate, which can be found in
\cite[Lemma~4.20]{TV06}:
\begin{equation}
  \frac{|B(\Gamma,\rho)|}{p}\geq\rho^{|\Gamma|}.
  \label{eq:bohr-size}
\end{equation}

\subsection{Directional Bohr kernels}

We use symmetrized Bohr kernels to construct the structured
components in the regularity decomposition. The kernel estimates below control their
Fourier multipliers and their change under translations.

For functions $\phi,\psi\colon\F_p\to\mathbb C$, define the normalized
convolution and involution by
\[
  (\phi*\psi)(t):=\E_{s\in\F_p}\phi(s)\psi(t-s),
  \qquad
  \widetilde\phi(t):=\overline{\phi(-t)}.
\]
For $v\in G$, $b\in\F_p$, and $h\colon G\to\mathbb C$, put
\[
  T_\phi^v h(\bx):=\E_{t\in\F_p}\phi(t)h(\bx-tv),
  \qquad
  \tau_b\phi(t):=\phi(t+b),
  \qquad
  \tau_b^v h(\bx):=h(\bx+bv).
\]
For a real number $u$, put $(u)_+:=\max\{u,0\}$. Let
$\varnothing\neq\Gamma\subseteq\F_p\setminus\{0\}$ and
$0<\rho\leq1/4$. Following Tao's smoothing approach~\cite{Tao14},
we use the normalized Lipschitz cutoff associated with $B(\Gamma,\rho)$:
\[
  \mu_{\Gamma,\rho}(t)
  :=
  \frac{
    \left(
      1-\rho^{-1}\max_{\xi\in\Gamma}
      \left\|\frac{\xi t}{p}\right\|_{\T}
    \right)_+
  }{
    \displaystyle
    \E_{s\in\F_p}
    \left(
      1-\rho^{-1}\max_{\xi\in\Gamma}
      \left\|\frac{\xi s}{p}\right\|_{\T}
    \right)_+
  }.
\]
For $\Gamma=\varnothing$, we set $\mu_{\varnothing,\rho}=1$.  For fixed
$\Gamma$ and $\rho$, write $\mu:=\mu_{\Gamma,\rho}$.
The structured components will be formed using the symmetrized
kernel $\mu*\widetilde\mu$. The following lemma records the Fourier
and translation estimates needed for these averages.

\begin{lemma}
\label{lem:bohr-kernel-properties}
Let $\mu:=\mu_{\Gamma,\rho}$ for some $\Gamma\subseteq\F_p\setminus\{0\}$ and $0<\rho\leq1/4$.
Then, the functions $\mu$ and $\mu*\widetilde\mu$ are nonnegative and have mean
one.  Moreover, we have
\begin{align}
  \|\mu\|_{\mathcal L^2(\F_p)}
  &\ll_{|\Gamma|,\rho}1,
  \label{eq:qualitative-kernel-L2}\\
  \widehat{\mu*\widetilde\mu}(\xi)
  &=|\widehat\mu(\xi)|^2,
  \label{eq:qualitative-kernel-multiplier}\\
  |1-\widehat\mu(\xi)|
  &\ll\rho
  \qquad(\xi\in\Gamma\cup\{0\}).
  \label{eq:qualitative-low-frequency}
\end{align}
For $0< \sigma \leq 1/2$,
if $b\in B(\Gamma,\sigma)$, then
\begin{equation}
  \|\tau_b(\mu*\widetilde\mu)-\mu*\widetilde\mu
  \|_{\mathcal L^1(\F_p)}
  \ll_{|\Gamma|,\rho}\sigma.
  \label{eq:qualitative-kernel-translation}
\end{equation}

\end{lemma}

\begin{proof}[Proof of Lemma~\ref{lem:bohr-kernel-properties}]
If $\Gamma=\varnothing$, then $\mu=1$ and all assertions are immediate.
Now, we suppose that $\Gamma\neq\varnothing$. The cutoff before normalization
takes values in $[0,1]$ and is at least $1/2$ on
$B(\Gamma,\rho/2)$. By \eqref{eq:bohr-size}, its normalized mean is
bounded below by a positive constant depending only on $|\Gamma|$
and $\rho$, and this gives \eqref{eq:qualitative-kernel-L2}.

By construction, $\mu$ is nonnegative and has mean one, and the same
holds for $\mu*\widetilde\mu$. The convolution identity for the Fourier
transform gives \eqref{eq:qualitative-kernel-multiplier}.

For $\xi\in\Gamma\cup\{0\}$ and $t\in\operatorname{supp}\mu$, we have
$\|\xi t/p\|_{\T}\leq\rho$, and hence
$|1-e_p(-\xi t)|\leq2\pi\rho$.
Since $\mu$ is nonnegative and has mean one, it follows that
\[
  |1-\widehat\mu(\xi)|
  \leq
  \E_{t\in\F_p}\mu(t)|1-e_p(-\xi t)|
  \leq2\pi\rho.
\]
This proves \eqref{eq:qualitative-low-frequency}.

Finally, let $b\in B(\Gamma,\sigma)$. The triangle inequality on
$\T$ implies that the unnormalized cutoff changes by at most
$\sigma/\rho$ under translation by $b$. Using the lower bound
for the normalizing denominator, we obtain
\[
  \|\tau_b\mu-\mu\|_{\mathcal L^1(\F_p)}
  \ll_{|\Gamma|,\rho}\sigma.
\]
Since $\|\widetilde\mu\|_{\mathcal L^1(\F_p)}=1$, Young's inequality
therefore gives
\[
  \|\tau_b(\mu*\widetilde\mu)-\mu*\widetilde\mu
    \|_{\mathcal L^1(\F_p)}
  \leq
  \|\tau_b\mu-\mu\|_{\mathcal L^1(\F_p)}
  \ll_{|\Gamma|,\rho}\sigma,
\]
and this proves \eqref{eq:qualitative-kernel-translation}.
\end{proof}

We now express these kernel properties in terms of directional
averaging operators. The resulting identities will be used to compare
successive averages and control their changes under translation.

Let $\Gamma,\Gamma'\subseteq\F_p\setminus\{0\}$,
$0<\rho,\rho'\leq1/4$, $b,\xi\in\F_p$, and let
$v\in G\setminus\{0\}$.  Write
\[
  \mu:=\mu_{\Gamma,\rho},
  \qquad
  \mu':=\mu_{\Gamma',\rho'}.
\]
If $h\colon G\to\mathbb C$, then
\begin{equation}
  \widehat{T_{\mu*\widetilde\mu}^v h}(\bx;v;\xi)
  =|\widehat\mu(\xi)|^2\widehat h(\bx;v;\xi),
  \label{eq:smoothing-multiplier}
\end{equation}
and fibrewise Parseval gives
\begin{align}
  &\|T_{\mu'*\widetilde{\mu'}}^v h
      -T_{\mu*\widetilde\mu}^v h
    \|_{\mathcal L^2(G)}^2
  \notag\\
  &\qquad=
  \sum_{\xi\in\F_p}
  \E_{\bx\in G}|\widehat h(\bx;v;\xi)|^2
  \left(
    |\widehat{\mu'}(\xi)|^2-|\widehat\mu(\xi)|^2
  \right)^2.
  \label{eq:smoothing-parseval}
\end{align}
The translation estimate passes to the smoothed function:
if $0\leq h\leq1$, then
\begin{equation}
  \|\tau_b^vT_{\mu*\widetilde\mu}^vh
       -T_{\mu*\widetilde\mu}^vh
  \|_{\mathcal L^\infty(G)}
  \leq
  \|\tau_b(\mu*\widetilde\mu)-\mu*\widetilde\mu
  \|_{\mathcal L^1(\F_p)}.
  \label{eq:smoothing-invariance}
\end{equation}
Indeed, a change of variables gives
$\tau_b^vT_{\mu*\widetilde\mu}^vh
=T_{\tau_b(\mu*\widetilde\mu)}^vh$,
and the bound follows from $|h|\leq1$.

\section{Proof of Theorem~\ref{thm:main}}
\label{sec:simultaneous-proof}

Following the approach developed in~\cite{CDH26}, we construct a
regularity decomposition in each prescribed direction.
The generalized von Neumann estimate controls the pseudorandom
contributions. A Bohr-set estimate provides many parameters for
which the structured components are almost invariant under the
corresponding polynomial shifts, and a convolution inequality
gives the required lower bounds for their contributions.
We then select one nonzero parameter that works for every
choice of $\omega$.

\subsection{A generalized von Neumann estimate}

We first establish an $\mathcal L^2(\F_p)$ estimate for the
polynomial counting averages in terms of the mixed directional
box norms. This estimate will allow us to discard the
pseudorandom contributions after applying the regularity
decomposition.
We begin with Kuca's finite-field multidimensional polynomial
Szemer\'edi theorem~\cite[Theorem~1.2]{Kuca24}.

\begin{theorem}
\label{thm:kuca}
Let $D,r\in \N$. Let $\mathcal{W}=\{w_1,\ldots,w_r\in\Z^D\setminus\{0\}\}$, and $\mathcal{R}=\{P_1,\ldots,P_r\in\Z[t]\}$ be a
collection of linearly independent polynomials with zero constant terms. There exist $p_1=p_1(\mathcal{W},\mathcal{R})$, and $c=c(\mathcal{R})>0$ such that for every prime $p \geq p_1$, we have 
\[
  \E_{\bx\in\F_p^D}\E_{d\in\F_p}
  f_0(\bx)\prod_{i=1}^r
  f_i\bigl(\bx+P_i(d)w_i\bigr)
  =
  \E_{\bx\in\F_p^D}f_0(\bx)
  \prod_{i=1}^r
  \E_{n_i\in\F_p}f_i(\bx+n_iw_i)
  +O_{\mathcal{R}}(p^{-c}),
\]
for all $1$-bounded functions
$f_0,\ldots,f_r\colon\F_p^D\to\mathbb C$.
\end{theorem}

The following consequence applies uniformly to every
subconfiguration. The bound can be obtained from any one of the
functions $g_i$, $i\in S$.

\begin{lemma}
\label{lem:directional-vn}
Let $D,k\in \N$, and denote $G=\F_p^D$. Let $\mathcal{V}=\{v_1,\ldots,v_k\in\Z^D\setminus\{0\}\}$, and $\mathcal{P}=\{P_1,\ldots,P_k\in\Z[t]\}$ be a
collection of linearly independent polynomials with zero constant terms. There exist $p_0=p_0(\mathcal{V},\mathcal{P})$, and $c=c(\mathcal{P})>0$ such that for every prime $p \geq p_0$, we have
\begin{align}
  \E_{d\in\F_p}
  \left|
    \E_{\bx\in G}g_0(\bx)
    \prod_{j\in S}g_j\bigl(\bx+P_j(d)v_j\bigr)
  \right|^2
  &\leq
  \min_{i\in S}\|g_i\|_{G,\langle v_i\rangle}^{2}
  +O_{\cP}(p^{-c})
  \label{eq:vn-box}\\
  &\leq
  \min_{i\in S}
  \left(
    \sup_{\xi\in\F_p}
    \E_{\bx\in G}|\widehat g_i(\bx;v_i;\xi)|^2
  \right)^{1/2}
  +O_{\cP}(p^{-c}),
  \label{eq:vn-fourier}
\end{align}
for every nonempty $S\subseteq\{1,\ldots,k\}$ and all $1$-bounded
functions $g_0\colon G\to\mathbb C$ and
$g_i\colon G\to\mathbb C$ for $i\in S$.
\end{lemma}

\begin{proof}[Proof of Lemma~\ref{lem:directional-vn}]
Expanding the square and writing the second averaging variable as $\bx+\bh$ gives
\begin{align*}
  \E_{d\in\F_p}
  \left|
    \E_{\bx\in G}g_0(\bx)
    \prod_{j\in S}g_j\bigl(\bx+P_j(d)v_j\bigr)
  \right|^2
  &=
  \E_{\bx,\bh\in G}\E_{d\in\F_p}
  \triangle_{\bh}g_0(\bx)
  \prod_{j\in S}
  \triangle_{\bh}g_j\bigl(\bx+P_j(d)v_j\bigr).
\end{align*}
Since there are only finitely many subfamilies, we fix one exponent $c>0$,
one implicit constant, and one prime threshold that work for all of them.
For each fixed $\bh$, apply Theorem~\ref{thm:kuca} to the subfamily indexed
by $S$.  Its error is uniform in $\bh$, since every multiplicative
derivative remains $1$-bounded.  We obtain
\[
  \E_{d\in\F_p}
  \left|
    \E_{\bx\in G}g_0(\bx)
    \prod_{j\in S}g_j\bigl(\bx+P_j(d)v_j\bigr)
  \right|^2
  =
  \E_{\bx,\bh}
  \triangle_{\bh}g_0(\bx)
  \prod_{j\in S}
  \E_{n_j\in\F_p}
  \triangle_{\bh}g_j(\bx+n_jv_j)
  +O_{\cP}(p^{-c}).
\]
Fix $i\in S$.  Bounding all factors except the $i$th by $1$, and then using
Cauchy--Schwarz yields
\begin{align*}
  &\E_{d\in\F_p}
  \left|
    \E_{\bx\in G}g_0(\bx)
    \prod_{j\in S}g_j\bigl(\bx+P_j(d)v_j\bigr)
  \right|^2
  \\
  &\qquad\leq
  \left|
    \E_{\bx,\bh}
    \triangle_{\bh}g_0(\bx)
    \prod_{j\in S}
    \E_{n_j\in\F_p}
    \triangle_{\bh}g_j(\bx+n_jv_j)
  \right|
  +O_{\cP}(p^{-c})\\
  &\qquad\leq
  \left(
    \E_{\bx,\bh}
    \left|
      \E_{n_i\in\F_p}
      \triangle_{\bh}g_i(\bx+n_iv_i)
    \right|^2
  \right)^{1/2}
  +O_{\cP}(p^{-c})\\
  &\qquad=
  \|g_i\|_{G,\langle v_i\rangle}^{2}
  +O_{\cP}(p^{-c}).
\end{align*}
Taking
the minimum over $i$ proves
\eqref{eq:vn-box}, and Lemma~\ref{lem:directional-fourier-control} gives
\eqref{eq:vn-fourier}.
\end{proof}

\subsection{Directional arithmetic regularity decomposition}

The preceding estimate motivates a decomposition adapted to each
prescribed direction. Following Tao's arithmetic regularity argument~\cite{Tao14}, we
construct these decompositions with structured components defined
using symmetrized Bohr kernels. The directional Fourier
coefficients of the pseudorandom components are small in
$\mathcal L^2(G)$, uniformly in frequency, and the remaining
components have small $\mathcal L^2(G)$ norms.

\begin{proposition}
\label{prop:directional-regularity}
Let $D,k\in \N$, and denote $G=\F_p^D$.
Let $0<\varepsilon<1$ and let
$\eta_1,\eta_2\colon[1,\infty)\to[2,\infty)$ be nondecreasing growth
functions.  Let
$p$ be prime, let
$v_1,\ldots,v_k\in G\setminus\{0\}$, and let
$f\colon G\to[0,1]$.  Then there are finite sets
$\Gamma_i\subseteq\F_p\setminus\{0\}$, a radius
$\sigma\in(0,1/4]$, and parameters $\rho_1,\rho_2\in(0,1/2]$ such that
\begin{equation}
  \max_{1\leq i\leq k}|\Gamma_i|
  =O_{\varepsilon,k,\eta_1,\eta_2}(1),
  \qquad
  \rho_1\leq
  \frac{1}{
    \eta_1(\max_i|\Gamma_i|+\varepsilon^{-1})
  },
  \qquad
  \rho_2\leq
  \frac{1}{\eta_2(\rho_1^{-1})}.
  \label{eq:regularity-scales}
\end{equation}
The parameters $\rho_1,\rho_2$ and the radius $\sigma$ are bounded away
from zero in terms of $\varepsilon,k,\eta_1$, and $\eta_2$ alone.
For every $1\leq i\leq k$, put
$\mu_i:=\mu_{\Gamma_i,\sigma}$.  There is a decomposition
\[
  f=f_{\str}^{(i)}+f_{\psr}^{(i)}+f_{\sml}^{(i)}
\]
with the following properties:
\begin{enumerate}
  \item The functions $f_{\str}^{(i)}$, $f_{\psr}^{(i)}$, and
$f_{\sml}^{(i)}$ are all $1$-bounded;

  \item
  \[
    f_{\str}^{(i)}
    =T_{\mu_i*\widetilde\mu_i}^{v_i}f,
    \qquad
    0\leq f_{\str}^{(i)}\leq1,
    \qquad
    \E f_{\str}^{(i)}=\E f;
  \]

  \item if $b\in B(\Gamma_i,\rho_1)$, then
  \begin{equation}
    \|\tau_b^{v_i}f_{\str}^{(i)}
      -f_{\str}^{(i)}\|_{\mathcal L^\infty(G)}
    \leq\varepsilon;
    \label{eq:structured-invariance}
  \end{equation}

  \item
  \begin{equation}
    \sup_{\xi\in\F_p}
    \left(
      \E_{\bx\in G}
      |\widehat{f_{\psr}^{(i)}}(\bx;v_i;\xi)|^2
    \right)^{1/2}
    \leq\rho_2;
    \label{eq:pseudorandom-fourier}
  \end{equation}

  \item
  \begin{equation}
    \sum_{i=1}^k
    \|f_{\sml}^{(i)}\|_{\mathcal L^2(G)}^2
    \leq\varepsilon^2.
    \label{eq:synchronized-smallness}
  \end{equation}
\end{enumerate}
In particular,
\[
  \|f_{\psr}^{(i)}\|_{G,\langle v_i\rangle}^{2}
  \leq\rho_2.
\]
All bounds are independent of $p,f,D$, and the particular nonzero vectors
$v_i$.
\end{proposition}

\begin{proof}[Proof of Proposition~\ref{prop:directional-regularity}]
We construct nested frequency sets and smooth $f$ in each direction
using symmetrized Bohr kernels. A pigeonhole argument then selects
two consecutive stages for which the sum of the squared
$\mathcal L^2(G)$ norms of the differences between the corresponding
smoothings is small. These smoothings determine the three components.

We first recall that for every $1\leq i\leq k$, fibrewise Parseval gives
\begin{equation}
  \sum_{\xi\in\F_p}
  \E_{\bx\in G}|\widehat f(\bx;v_i;\xi)|^2
  =\|f\|_{\mathcal L^2(G)}^2
  \leq1.
  \label{eq:directional-energy-total}
\end{equation}

We construct the scales and frequency sets inductively, following
\cite{Tao14}. Choose an integer $L\geq2$ sufficiently large in terms
of $k$ and $\varepsilon$ that
$2k/(L-1)\leq\varepsilon^2/3$. Choose a sufficiently small absolute
constant $c_0>0$. We construct positive scales
\[
  \delta_0:=
  \min\left\{\frac14,c_0\frac{\varepsilon}{\sqrt{k}}\right\},
  \qquad
  \delta_0>\delta_1>\cdots>\delta_{3L}>0,
\]
together with nested frequency sets
\[
  \varnothing=\Gamma_{i,0}
  \subseteq\Gamma_{i,1}
  \subseteq\cdots\subseteq\Gamma_{i,L}
  \subseteq\F_p\setminus\{0\}
  \qquad(1\leq i\leq k).
\]
Each scale is chosen as a fixed positive function of the fixed inputs
and the previously chosen scales and cardinality bounds.

Set $R_0:=0$. For $0\leq j<L$, suppose that
$\delta_0,\ldots,\delta_{3j}$, the sets $\Gamma_{i,j}$, and a bound
$R_j$, depending only on $\varepsilon,k,\eta_1,\eta_2$, and $j$,
have been chosen, with
\[
  \max_i|\Gamma_{i,j}|\leq R_j.
\]
Define
\[
  \mu_{i,j}:=\mu_{\Gamma_{i,j},\delta_{3j}},
\]
and choose $\delta_{3j+1}\in(0,\delta_{3j})$ sufficiently
small that Lemma~\ref{lem:bohr-kernel-properties} gives translation error
at most $\varepsilon$ on $B(\Gamma_{i,j},\delta_{3j+1})$ whenever
$|\Gamma_{i,j}|\leq R_j$, and so that
\[
  \delta_{3j+1}
  \leq
  \eta_1(R_j+\varepsilon^{-1})^{-1}.
\]
Next choose positive scales
$\delta_{3j+3}<\delta_{3j+2}<\delta_{3j+1}$ sufficiently small that
\begin{equation}
  \delta_{3j+2}
  \ll
  \min\left\{
    \frac12,
    \eta_2(\delta_{3j+1}^{-1})^{-1}
  \right\}.
  \label{eq:uniformity-scale}
\end{equation}
When $j\geq1$, impose in addition that
\[
  k\delta_{3j+2}^4\leq\frac{\varepsilon^2}{3}.
\]
We include frequencies at which a directional Fourier coefficient
of $f$ has large $\mathcal L^2(G)$ norm, to control the pseudorandom
component. We also include frequencies at which an earlier kernel
has a large Fourier coefficient, to control the difference between
successive smoothings. We therefore set
\begin{align}
  \Gamma_{i,j+1}
  :=\Gamma_{i,j}
  &\cup
  \left\{
    \xi\neq0\colon
    \left(
      \E_{\bx\in G}|\widehat f(\bx;v_i;\xi)|^2
    \right)^{1/2}
    \geq\delta_{3j+2}
  \right\}
  \notag\\
  &\cup
  \bigcup_{\ell=0}^j
  \{\xi\neq0\colon
    |\widehat{\mu_{i,\ell}}(\xi)|
    \geq\delta_{3j+2}\}.
  \label{eq:frequency-selection}
\end{align}
The first new set has cardinality at most $\delta_{3j+2}^{-2}$ by
\eqref{eq:directional-energy-total}. For each $0\leq\ell\leq j$,
Lemma~\ref{lem:bohr-kernel-properties} and Plancherel bound the number
of frequencies satisfying
$|\widehat{\mu_{i,\ell}}(\xi)|\geq\delta_{3j+2}$ by
$O_{R_\ell,\delta_{3\ell}}(\delta_{3j+2}^{-2})$.
We may therefore choose
$R_{j+1}=O_{\varepsilon,k,\eta_1,\eta_2,j}(1)$ which bounds the
cardinalities of all the sets $\Gamma_{i,j+1}$, independently of
$f,p,D$, and the directions. Repeating this procedure constructs
the entire hierarchy. Since there are finitely many stages and
each scale depends only on the fixed inputs and the previously
chosen bounds and scales, all cardinalities and inverse scales
are bounded in terms of the fixed inputs alone.

We now choose the stage of the decomposition. The quantities
\[
  \sum_{i=1}^k
  \sum_{\xi\in\Gamma_{i,j}}
  \E_{\bx\in G}|\widehat f(\bx;v_i;\xi)|^2
\]
are nondecreasing in $j$ and lie in $[0,k]$.
Since the frequency sets are nested, each frequency is counted
at most twice in the sum below. Hence,
\[
  \sum_{j=0}^{L-2}
  \sum_{i=1}^k
  \sum_{\xi\in\Gamma_{i,j+2}\setminus\Gamma_{i,j}}
  \E_{\bx\in G}|\widehat f(\bx;v_i;\xi)|^2
  \leq2k.
\]
The choice of $L$ therefore gives an index $0\leq j\leq L-2$ such that
\begin{equation}
  \sum_{i=1}^k
  \sum_{\xi\in\Gamma_{i,j+2}\setminus\Gamma_{i,j}}
  \E_{\bx\in G}|\widehat f(\bx;v_i;\xi)|^2
  \leq\frac{\varepsilon^2}{3}.
  \label{eq:energy-band}
\end{equation}
The use of $\Gamma_{i,j+2}$ allows us to control both kernels
$\mu_{i,j}$ and $\mu_{i,j+1}$: this set contains every nonzero
frequency at which either kernel has a Fourier coefficient
of magnitude at least $\delta_{3j+5}$.

Fix this index $j$. We define the three components using the
smoothings at stages $j$ and $j+1$:
\begin{align}
  f_{\str}^{(i)}
  &:=T_{\mu_{i,j}*\widetilde\mu_{i,j}}^{v_i}f,
  \notag\\
  f_{\sml}^{(i)}
  &:=T_{\mu_{i,j+1}*\widetilde\mu_{i,j+1}}^{v_i}f
    -T_{\mu_{i,j}*\widetilde\mu_{i,j}}^{v_i}f,
  \notag\\
  f_{\psr}^{(i)}
  &:=f-T_{\mu_{i,j+1}*\widetilde\mu_{i,j+1}}^{v_i}f.
  \label{eq:selected-decomposition}
\end{align}
Since $0\leq f\leq1$ and the kernels are nonnegative and have mean
one, both smoothings in \eqref{eq:selected-decomposition} take
values in $[0,1]$. Hence the structured component takes values
in $[0,1]$, while the other two components take values in $[-1,1]$.
In particular, all three components are $1$-bounded.

\begin{enumerate}
  \item[1.] (\emph{The small component.})
  By \eqref{eq:smoothing-parseval}, we have
\[
  \|f_{\sml}^{(i)}\|_{\mathcal L^2(G)}^2
  =
  \sum_{\xi\in\F_p}
  \E_{\bx\in G}|\widehat f(\bx;v_i;\xi)|^2
  \left(
    |\widehat{\mu_{i,j+1}}(\xi)|^2
    -|\widehat{\mu_{i,j}}(\xi)|^2
  \right)^2.
\]
Split the frequency sum into
\[
  \{0\}\cup\Gamma_{i,j},
  \qquad
  \Gamma_{i,j+2}\setminus\Gamma_{i,j},
  \qquad
  \F_p\setminus(\Gamma_{i,j+2}\cup\{0\}).
\]
On the first range, both multipliers are $1+O(\delta_{3j})$ by
\eqref{eq:qualitative-low-frequency}. On the second range, their
squared difference is at most $1$, since both multipliers lie
in $[0,1]$. On the third range, the selection rule at stage $j+1$
gives
\[
  |\widehat{\mu_{i,j}}(\xi)|,
  |\widehat{\mu_{i,j+1}}(\xi)|
  <\delta_{3j+5},
\]
so the squared difference of the multipliers is at most
$\delta_{3j+5}^4$.
Summing over $i$ and using \eqref{eq:directional-energy-total}
and \eqref{eq:energy-band}, we obtain
\[
  \sum_{i=1}^k
  \|f_{\sml}^{(i)}\|_{\mathcal L^2(G)}^2
  \leq
  Ck\delta_{3j}^2
  +\frac{\varepsilon^2}{3}
  +k\delta_{3j+5}^4
  \leq\varepsilon^2,
\]
where $C$ is an absolute constant. The last inequality follows
from $\delta_{3j}\leq\delta_0$, the choice of $c_0$, and the
condition imposed at stage $j+1$.
This proves \eqref{eq:synchronized-smallness}.

  \item[2.] (\emph{The pseudorandom component.})
  By \eqref{eq:smoothing-multiplier}, we have
\[
  \widehat{f_{\psr}^{(i)}}(\bx;v_i;\xi)
  =
  \bigl(1-|\widehat{\mu_{i,j+1}}(\xi)|^2\bigr)
  \widehat f(\bx;v_i;\xi).
\]
If $\xi\notin\Gamma_{i,j+1}\cup\{0\}$, then we have
\[
  \left(
    \E_{\bx\in G}|\widehat f(\bx;v_i;\xi)|^2
  \right)^{1/2}
  <\delta_{3j+2}
\]
by \eqref{eq:frequency-selection}. If
$\xi\in\Gamma_{i,j+1}\cup\{0\}$, then
\eqref{eq:qualitative-low-frequency} gives
\[
  1-|\widehat{\mu_{i,j+1}}(\xi)|^2
  \ll \delta_{3j+3}.
\]
Since $\delta_{3j+3}<\delta_{3j+2}$,
\eqref{eq:pseudorandom-fourier} holds with
\[
  \rho_2:=C\delta_{3j+2},
\]
where $C$ is a sufficiently large absolute constant. This constant
is absorbed by the choice of scales in \eqref{eq:uniformity-scale}.

  \item[3.] (\emph{The structured component.})
  Set
\[
  \Gamma_i:=\Gamma_{i,j},
  \qquad
  \sigma:=\delta_{3j},
  \qquad
  \rho_1:=\delta_{3j+1},
  \qquad
  \mu_i:=\mu_{i,j}.
\]
The choice of $\delta_{3j+1}$ and
\eqref{eq:smoothing-invariance} give
\eqref{eq:structured-invariance}. Moreover, since
$\max_i|\Gamma_i|\leq R_j$ and $\eta_1$ is nondecreasing,
the choice of $\delta_{3j+1}$ gives
\[
  \rho_1=\delta_{3j+1}
  \leq\frac{1}{\eta_1(R_j+\varepsilon^{-1})}
  \leq\frac{1}{\eta_1(\max_i|\Gamma_i|+\varepsilon^{-1})}.
\]
The choice of $\rho_2$, together with
\eqref{eq:uniformity-scale}, gives
$\rho_2\leq\eta_2(\rho_1^{-1})^{-1}$.
Thus, both scale inequalities in \eqref{eq:regularity-scales} hold.
\end{enumerate}
Finally, since each kernel $\mu_i*\widetilde\mu_i$ has mean one,
we have $\E f_{\str}^{(i)}=\E f$. The bounds on $|\Gamma_i|$,
$\sigma^{-1}$, $\rho_1^{-1}$, and $\rho_2^{-1}$ follow from the
construction. Applying Lemma~\ref{lem:directional-fourier-control}
to the $1$-bounded function $f_{\psr}^{(i)}$ and using
\eqref{eq:pseudorandom-fourier} gives the claimed box-norm estimate.
\end{proof}

\begin{remark}
\label{rem:symmetrization(decomposition)}
In Tao's one-dimensional argument, the structured term is smoothed
with a cutoff $\mu$, whereas the $i$th directional component here uses
$\mu_i*\widetilde\mu_i$.  Replacing $\widehat\mu_i$ by
$|\widehat\mu_i|^2$ changes the regularity estimates only by fixed
absolute factors.  The symmetrized kernel is needed when estimating the structured
contribution, because
$T_{\mu_i*\widetilde\mu_i}^{v_i}
=T_{\mu_i}^{v_i}T_{\widetilde\mu_i}^{v_i}$
allows Lemma~\ref{lem:convolution-inequality} to handle
structured components from different directions.
\end{remark}

\subsection{Polynomial images in prescribed Bohr sets}

We want to find a common parameter $d$ for which every structured
component is almost invariant under its prescribed shift.
The following lemma counts the parameters satisfying
$P_i(d)\in B(\Gamma_i,r_i)$ for all $i$ simultaneously.

\begin{lemma}
\label{lem:polynomial-bohr}
Let $k\in \N$ and $\mathcal{P}=\{P_1,\ldots,P_k\in\Z[t]\}$ be a
collection of linearly independent polynomials with zero constant terms.
For $1\leq i\leq k$, let
$\Gamma_i\subseteq\F_p\setminus\{0\}$ and $0<r_i\leq1/2$.  Then we have
\[
  \E_{d\in\F_p}
  \prod_{i=1}^k\1_{B(\Gamma_i,r_i)}(P_i(d))
  =
  \prod_{i=1}^k\frac{|B(\Gamma_i,r_i)|}{p}
  +O_{\cP,\{|\Gamma_i|\}_{i=1}^k}
  \left(
    p^{-1/2}(\log p)^{\sum_{i=1}^k|\Gamma_i|}
  \right).
\]
\end{lemma}

\begin{proof}[Proof of Lemma~\ref{lem:polynomial-bohr}]
The proof is analogous to that of \cite[Lemma~3.4]{CDH26}
and is omitted.
\end{proof}

\subsection{A convolution inequality}

Almost invariance allows us to remove the polynomial shifts from
the structured contribution at the cost of a small error. The remaining
factors are different directional smoothings of the same function.
The following convolution inequality gives the required lower bound
for the average of their product with the original function.

\begin{lemma}
\label{lem:convolution-inequality}
Let $m\geq0$, let $v_1,\ldots,v_m\in \F_p^D$, and let
$f\colon \F_p^D\to[0,\infty)$ satisfy $\E f=\alpha$.  For every
$1\leq i\leq m$, let $\mu_i\colon\F_p\to[0,\infty)$ satisfy
$\E\mu_i=1$.  Then we have
\begin{equation}
  \E_{\bx\in \F_p^D}
  f(\bx)\prod_{i=1}^mT_{\mu_i*\widetilde\mu_i}^{v_i}f(\bx)
  \geq\alpha^{m+1}.
  \label{eq:convolution-inequality}
\end{equation}
\end{lemma}

\begin{proof}[Proof of Lemma~\ref{lem:convolution-inequality}]
The cases $m=0$ and $\alpha=0$ are immediate. By homogeneity,
we may therefore assume that $m\geq1$ and $\E f=1$.
We first treat the case in which $f$ is strictly positive.

We estimate each factor separately after taking logarithms.
Fix $1\leq i\leq m$. Since $\mu_i$ is nonnegative and has mean
one, Jensen's inequality gives
\[
  \log\bigl(T_{\mu_i*\widetilde\mu_i}^{v_i}f(\bx)\bigr)
  =
  \log\bigl(T_{\mu_i}^{v_i}
    T_{\widetilde\mu_i}^{v_i}f(\bx)\bigr)
  \geq
  \Bigl[
    T_{\mu_i}^{v_i}
    \bigl(\log(T_{\widetilde\mu_i}^{v_i}f)\bigr)
  \Bigr](\bx).
\]
Multiplying by $f$ and averaging over $\bx$, we obtain
\begin{align*}
  \E f\log\bigl(T_{\mu_i*\widetilde\mu_i}^{v_i}f\bigr)
  &\geq
  \E f\,T_{\mu_i}^{v_i}
  \bigl(\log(T_{\widetilde\mu_i}^{v_i}f)\bigr)\\
  &=
  \E (T_{\widetilde\mu_i}^{v_i}f)
  \log(T_{\widetilde\mu_i}^{v_i}f),
\end{align*}
where the last equality follows by a change of variables.
The last expression is nonnegative by Jensen's inequality
for the convex function $u\mapsto u\log u$, since
$\E T_{\widetilde\mu_i}^{v_i}f=1$.

We now combine these bounds. Since $f$ is nonnegative and
$\E f=1$, Jensen's inequality with weights $f$ gives
\begin{align*}
  \log\left(
    \E f\prod_{i=1}^mT_{\mu_i*\widetilde\mu_i}^{v_i}f
  \right)
  &\geq
  \E f\log\left(
    \prod_{i=1}^mT_{\mu_i*\widetilde\mu_i}^{v_i}f
  \right)\\
  &=
  \sum_{i=1}^m
  \E f\log\bigl(T_{\mu_i*\widetilde\mu_i}^{v_i}f\bigr)\\
  &\geq0.
\end{align*}
Exponentiating proves the result when $f$ is strictly positive
and has mean one. For a general nonnegative function of mean one,
apply the preceding argument to $(f+\delta)/(1+\delta)$ and let
$\delta\downarrow0$.
This proves \eqref{eq:convolution-inequality}.
\end{proof}

\begin{remark}
\label{rem:symmetrization}
The symmetrization is essential in general: inequality
\eqref{eq:convolution-inequality} need not hold with
$T_{\mu_i}^{v_i}$ in place of
$T_{\mu_i*\widetilde\mu_i}^{v_i}$.
Indeed, take $D=m=1$, $v_1=1$, and set
\[
  f=\1_{\{0\}},
  \qquad
  \mu=\frac p2\bigl(\1_{\{1\}}+\1_{\{-1\}}\bigr).
\]
Then $\mu$ is nonnegative and has mean one, but
\[
  \E_{x\in\F_p} f(x)T_\mu^1f(x)
  =0
  <\frac{1}{p^2}
  =(\E f)^2.
\]
\end{remark}

\subsection{Completion of the proof}

We now combine the preceding results to obtain the required
lower bound for every choice of $\omega$ with the same nonzero
parameter $d$. We first find many parameters satisfying the
Bohr conditions and then remove those belonging to the
exceptional sets.

\begin{proof}[Proof of Theorem~\ref{thm:main}]
Recall that $\alpha=\E_{\bx\in G}\1_A(\bx)$.
We may assume that $0<\varepsilon<1$. Fix
\[
  0<\vartheta\ll_k\varepsilon.
\]
Choose nondecreasing growth functions $\eta_1,\eta_2$, depending only
on $\vartheta,\cP$, and $k$, such that $\eta_1(t)\geq t$ and
$\eta_2(t)^{-1}\ll_{\cP,k}\vartheta^2t^{-kt}$, with the implicit
constant sufficiently small.

Apply Proposition~\ref{prop:directional-regularity} to $\1_A$, with
parameter $\vartheta$ and growth functions $\eta_1,\eta_2$.
We obtain $\Gamma_i$, parameters $\rho_1,\rho_2$, normalized
nonnegative kernels $\mu_i$, and decompositions in the prescribed directions
\[
  \1_A=f_{\str}^{(i)}+f_{\psr}^{(i)}+f_{\sml}^{(i)}
  \qquad(1\leq i\leq k).
\]
In particular, we have
\begin{gather*}
  f_{\str}^{(i)}
  =T_{\mu_i*\widetilde\mu_i}^{v_i}\1_A,
  \qquad
  0\leq f_{\str}^{(i)}\leq1,\\
  \|f_{\psr}^{(i)}\|_{G,\langle v_i\rangle}^{2}
  \leq\rho_2,
  \qquad
  \sum_{i=1}^k\|f_{\sml}^{(i)}\|_{\mathcal L^2(G)}^2
  \leq\vartheta^2,
  \qquad
  \|\tau_b^{v_i}f_{\str}^{(i)}-f_{\str}^{(i)}\|_{\mathcal L^\infty(G)}
  \leq\vartheta
  \quad\bigl(b\in B(\Gamma_i,\rho_1)\bigr).
\end{gather*}
By \eqref{eq:regularity-scales} and the choice of $\eta_1$,
$\max_i|\Gamma_i|\leq\rho_1^{-1}$, so
$\sum_{i=1}^k|\Gamma_i|\leq k\rho_1^{-1}$.
Since $0<\rho_1<1$, the choice of $\eta_2$ therefore gives
\begin{equation}
  \rho_2
  \leq\eta_2(\rho_1^{-1})^{-1}
  \ll_{\cP,k}
  \vartheta^2\rho_1^{k\rho_1^{-1}}
  \leq
  \vartheta^2\rho_1^{\sum_{i=1}^k|\Gamma_i|}.
  \label{eq:rho-hierarchy}
\end{equation}
Proposition~\ref{prop:directional-regularity} also gives bounds for
$\max_i|\Gamma_i|$ and $\rho_1^{-1}$ which depend only on
$\vartheta,\cP$, and $k$, and are independent of $p$, $A$, $D$,
and the directions.

We first collect the parameters for which every structured component
is almost invariant under its prescribed shift. Set
\[
  B^*:=
  \left\{
    d\in\F_p\colon
    P_i(d)\in B(\Gamma_i,\rho_1)
    \text{ for every }1\leq i\leq k
  \right\}.
\]
By \eqref{eq:bohr-size}, we have
\[
  \prod_{i=1}^k\frac{|B(\Gamma_i,\rho_1)|}{p}
  \geq\rho_1^{\sum_{i=1}^k|\Gamma_i|}.
\]
Since the ranks of $B(\Gamma_i,\rho_1)$ and $\rho_1^{-1}$
are bounded independently of $p$,
Lemma~\ref{lem:polynomial-bohr} implies that, for all
sufficiently large primes, we have
\begin{equation}
  \frac{|B^*|}{p}\geq\frac12\rho_1^{\sum_{i=1}^k|\Gamma_i|}.
  \label{eq:Bstar-density}
\end{equation}

We next remove parameters for which a term containing a pseudorandom
component is too large. Terms containing a small component will be
bounded uniformly in $d$ using their $\mathcal L^2(G)$ norms.
We therefore exclude parameters only for terms containing a
pseudorandom component and no small component.
Let $\varnothing\neq S\subseteq\{1,\ldots,k\}$.
For each $i\in S$, let $g_i$ be either $f_{\str}^{(i)}$ or
$f_{\psr}^{(i)}$, and suppose that $g_j=f_{\psr}^{(j)}$
for some $j\in S$.
Applying Lemma~\ref{lem:directional-vn} to $g_j$ gives
\[
  \E_{d\in\F_p}
  \left|
    \E_{\bx\in G}\1_A(\bx)
    \prod_{i\in S}g_i\bigl(\bx+P_i(d)v_i\bigr)
  \right|^2
  \ll_{\cP}\rho_2+p^{-c}.
\]
Markov's inequality yields
\[
  \frac1p
  \left|
    \left\{
      d\colon\left|
        \E_{\bx\in G}\1_A(\bx)
        \prod_{i\in S}g_i\bigl(\bx+P_i(d)v_i\bigr)
      \right|>\vartheta
    \right\}
  \right|
  \ll_{\cP}
  \vartheta^{-2}(\rho_2+p^{-c}).
\]
There are $O_k(1)$ choices of $S$ and of the functions $g_i$.
Let $E$ be the set of parameters for which at least one of
these averages has modulus greater than $\vartheta$.
By the union bound,
$|E|/p\ll_{\cP,k}\vartheta^{-2}(\rho_2+p^{-c})$.
Using \eqref{eq:rho-hierarchy} to control the first term
and taking $p$ sufficiently large to control the second,
we obtain
\begin{equation}
  \frac{|E|}{p}\leq\frac14\rho_1^{\sum_{i=1}^k|\Gamma_i|}.
  \label{eq:exceptional-density}
\end{equation}
Combining \eqref{eq:Bstar-density} and \eqref{eq:exceptional-density}
gives
\[
  \frac{|B^*\setminus E|}{p}
  \geq\frac14\rho_1^{\sum_{i=1}^k|\Gamma_i|}.
\]
The right-hand side is bounded below by a positive constant
depending only on $\vartheta,\cP$, and $k$.
Hence, for all sufficiently large primes, $|B^*\setminus E|>1$,
and we may choose
\begin{equation}
  d\in(B^*\setminus E)\setminus\{0\}.
  \label{eq:chosen-d}
\end{equation}

Having fixed $d$, we verify the required bound for every
choice of $\omega$. Fix $\omega\in\{0,1\}^k$ and put
$S:=\{i\colon\omega_i=1\}$. If $S=\varnothing$, then
\[
  I_{(\cP,\cV)_\omega}(\1_A,\ldots,\1_A)(d)=\E\1_A=\alpha.
\]
This gives the required bound when $S=\varnothing$.
Suppose now that $S\neq\varnothing$.
Leave $\1_A(\bx)$ undecomposed and apply the $i$th directional
decomposition to $\1_A(\bx+P_i(d)v_i)$:
\begin{equation}
  I_{(\cP,\cV)_\omega}(\1_A,\ldots,\1_A)(d)
  =
  \E_{\bx\in G}\1_A(\bx)
  \prod_{i\in S}
  \bigl(f_{\str}^{(i)}+f_{\psr}^{(i)}+f_{\sml}^{(i)}\bigr)
  \bigl(\bx+P_i(d)v_i\bigr).
  \label{eq:final-expansion}
\end{equation}

We first consider the term containing only structured components.
Since $d\in B^*$, we have
$P_i(d)\in B(\Gamma_i,\rho_1)$ for every $i$.
Telescoping the product and using $1$-boundedness gives
\begin{align*}
  \left|
    \E_{\bx\in G}\1_A(\bx)
    \prod_{i\in S}
    f_{\str}^{(i)}\bigl(\bx+P_i(d)v_i\bigr)
    -\E_{\bx\in G}\1_A(\bx)\prod_{i\in S}f_{\str}^{(i)}(\bx)
  \right|
  &\leq
  \sum_{i\in S}
  \|\tau_{P_i(d)}^{v_i}f_{\str}^{(i)}-f_{\str}^{(i)}\|_{\mathcal L^\infty(G)}
  \\
  &\leq |S|\vartheta
  \leq k\vartheta.
\end{align*}
Applying Lemma~\ref{lem:convolution-inequality} to $f=\1_A$
and the kernels $\mu_i$, $i\in S$, gives
\[
  \E_{\bx\in G}\1_A(\bx)
  \prod_{i\in S}f_{\str}^{(i)}(\bx)
  \geq\alpha^{|S|+1}.
\]
Combining these two estimates gives
\begin{equation}
  \E_{\bx\in G}\1_A(\bx)
  \prod_{i\in S}
  f_{\str}^{(i)}\bigl(\bx+P_i(d)v_i\bigr)
  \geq\alpha^{|S|+1}-k\vartheta.
  \label{eq:structured-main-lower-bound}
\end{equation}

Every other term in \eqref{eq:final-expansion} contains a small or a
pseudorandom component. If a term contains $f_{\sml}^{(j)}$, then,
by Cauchy--Schwarz and $1$-boundedness of all other factors,
its modulus is at most
\[
  \|f_{\sml}^{(j)}\|_{\mathcal L^1(G)}
  \leq\|f_{\sml}^{(j)}\|_{\mathcal L^2(G)}
  \leq\vartheta.
\]
If it contains no small component but contains a pseudorandom component,
then it is one of the terms used to define $E$, and its modulus is at most
$\vartheta$ because $d\notin E$.
There are at most $3^{|S|}-1\leq3^k$ terms other than the term
containing only structured components.
Combining these estimates with
\eqref{eq:structured-main-lower-bound}, we obtain
\[
  I_{(\cP,\cV)_\omega}(\1_A,\ldots,\1_A)(d)
  \geq
  \alpha^{|S|+1}-O_k(\vartheta)
  \geq
  \alpha^{|S|+1}-\varepsilon.
\]
The parameter $d$ in \eqref{eq:chosen-d} is independent of $\omega$, so
this holds simultaneously for all $\omega\in\{0,1\}^k$.
\end{proof}

\section*{Acknowledgements}

The author thanks David Conlon and Dingding Dong for helpful
discussions on popular polynomial differences that motivated
the present work.

\section*{Statement of AI use}

The main ideas of this paper, including the directional arithmetic
regularity decomposition in
Proposition~\ref{prop:directional-regularity}, were developed by
the author. After the initial argument was completed, GPT 5.6 Sol
pointed out the need to symmetrize the smoothing kernels defining
the structured components. This modification enables the lower
bound in Lemma~\ref{lem:convolution-inequality}; see also
Remarks~\ref{rem:symmetrization(decomposition)} and \ref{rem:symmetrization}.
GPT-6 Astra was subsequently used to assist with proofreading,
polish the writing, and improve the exposition.

\end{document}